\documentclass[11pt,a4paper]{article}
\usepackage[affil-it]{authblk}
\usepackage[utf8]{inputenc}
\usepackage{amsmath, amssymb, mathrsfs, amsthm, amsfonts}
\usepackage{latexsym}
\usepackage[pdftex]{hyperref}
\usepackage{a4wide}
\usepackage{color}
\usepackage{graphicx}

\title{The Varchenko Determinant of a Coxeter Arrangement}

\author{Götz Pfeiffer \thanks{G. Pfeiffer\\ School of Mathematics, Statistics and Applied Mathematics, National University of Ireland Galway, University Road, Galway, Ireland\\ e-mail: \texttt{goetz.pfeiffer@nuigalway.ie}}
, 
Hery Randriamaro \thanks{H. Randriamaro (Corresponding Author)\\ Mathematics Group, International Centre for Theoretical Physics, Strada Costiera 11, 34151 Trieste, Italy \\ e-mail: \texttt{hery.randriamaro@outlook.com} \\ This research was supported through the programme "Oberwolfach Leibniz Fellows" by the Mathematisches Forschungsinstitut Oberwolfach in 2017}}

\newtheorem{theorem}{Theorem}[section]

\newtheorem{lemma}[theorem]{Lemma}
\newtheorem{proposition}[theorem]{Proposition}

\theoremstyle{definition}

\theoremstyle{remark}

\begin{document}

\maketitle

\begin{abstract}
\noindent The Varchenko determinant is the determinant of a matrix defined from an arrangement of hyperplanes. Varchenko proved that this determinant has a beautiful factorization. It is, however, not possible to use this factorization to compute a Varchenko determinant from a certain level of complexity. Precisely at this point, we provide an explicit formula of this determinant for the hyperplane arrangements associated to the finite Coxeter groups. The intersections of hyperplanes with the chambers of such arrangements have nice properties which play a central role for the calculation of their relating determinants.

\bigskip 


\smallskip

\end{abstract}

\section{Introduction}

\noindent \noindent Let $x = (x_1, \dots, x_n)$ be a variable of the Euclidean space $\mathbb{R}^n$, and $a_1, \dots, a_n, b$ real coefficients such that $(a_1, \dots, a_n) \neq (0, \dots, 0)$. A hyperplane $H$ of $\mathbb{R}^n$ is a $(n-1)$--dimensional affine subspace $H := \{ x \in \mathbb{R}^n \ |\ a_1 x_1 + \dots + a_n x_n = b \}$. An arrangement of hyperplanes in $\mathbb{R}^n$ is a finite set of hyperplanes. For example, the most known hyperplane arrangement is certainly $\mathcal{A}_{A_{n-1}} = \big\{\{x \in \mathbb{R}^n \ |\ x_i-x_j=0\}\big\}_{1 \leq i < j \leq n}$ associated to the Coxeter group $A_{n-1}$.

\smallskip

\noindent A \emph{chamber} of a hyperplane arrangement $\mathcal{A}$ is a connected component of the complement $\mathbb{R}^n \setminus \bigcup_{H \in \mathcal{A}} H$. Denote the set of all chambers of $\mathcal{A}$ by $\mathfrak{C}(\mathcal{A})$.\\
Assign a variable $a_H$ to each hyperplane $H$ of an arrangement $\mathcal{A}$. Let $R_{\mathcal{A}} = \mathbb{Z}[a_H \,|\, H \in \mathcal{A}]$ be the ring of polynomials in variables $a_H$. The module of $R_{\mathcal{A}}$-linear combinations of chambers of the hyperplane arrangement $\mathcal{A}$ is
$$M_{\mathcal{A}} := \{ \sum_{C \in \mathfrak{C}(\mathcal{A})} x_C C\ |\ x_C \in R_{\mathcal{A}}\}.$$

\noindent Let $\mathcal{H}(C,D)$ be the set of hyperplanes separating the chambers $C$ and $D$ in $\mathfrak{C}(\mathcal{A})$. The $R_{\mathcal{A}}$-bilinear symmetric form $\mathsf{B}: M_{\mathcal{A}} \times M_{\mathcal{A}} \rightarrow R_{\mathcal{A}}$ on the hyperplane arrangement $\mathcal{A}$ defined by Varchenko \cite{Va} is $$\mathsf{B}(C,C):=1,\ \text{and}\ \mathsf{B}(C,D):= \prod_{H \in \mathcal{H}(C,D)} a_H\ \text{if}\ C \neq D.$$

\noindent The Varchenko matrix of the hyperplane arrangement $\mathcal{A}$ is the matrix $\big(\mathsf{B}(C,D)\big)_{C,D \in \mathfrak{C}(\mathcal{A})}$ associated to the bilinear symmetric form $\mathsf{B}$. In terms of Markov chains, it is the matrix of random walks on $\mathfrak{C}(\mathcal{A})$ whose walk probability from the chamber $C$ to the chamber $D$ is equal to $\mathsf{B}(C,D)$. The \emph{Varchenko determinant} of the hyperplane arrangement $\mathcal{A}$ is the determinant
$$\det \mathcal{A} := \det \big(\mathsf{B}(C,D)\big)_{C,D \in \mathfrak{C}(\mathcal{A})}.$$

\noindent One of the first appearances of this bilinear form was in the work of Schechtman and Varchenko \cite[1. Quantum groups]{ScVa}, in the implicit form of a symmetric bilinear form on a Verma module over a $\mathbb{C}$-algebra. It appeared more explicitly later, one year before its publication, however as a very special case, when Zagier studied a Hilbert space $\mathbb{H}$ together with a nonzero distinguished vector $|0\rangle$, and a collection of operators $a_k:\mathbb{H} \rightarrow \mathbb{H}$ satisfying the commutation relations $a(l)a^{\dag}(k) - qa^{\dag}(k)a(l) = \delta_{k,l}$, and the relation $a(k)|0\rangle = 0$. To demonstrate the realizability of its model, he defined an inner product space $\big(\mathbb{H}(q),\langle\cdot,\cdot\rangle\big)$ with basis $B$ consisting of $n$-particle states $a^{\dag}(k_1) \dots a^{\dag}(k_n)|0\rangle$, and proved that \cite[Theorem 1, 2]{Za} $$\det \big(\langle u,v \rangle\big)_{u,v \in B} = \prod_{k=1}^{n-1} (1- q^{k^2+k})^{\frac{n!(n-k)}{k^2+k}}.$$
It is the Varchenko determinant of $\mathcal{A}_{A_{n-1}}$ with all hyperplanes weighted by $q$. Using the diagonal solutions of the Yang-Baxter equation, Duchamp et al. computed \cite[6.4.2 A Decomposition of $B_n$]{DuEtAl} $$\det \mathcal{A}_{A_{n-1}} = \prod_{\substack{I \in 2^{[n]} \\ |I| \geq 2}} \Big(1- \prod_{\{i,j\} \in \binom{I}{2}} a_{H_{i,j}}^2\Big)^{(|I|-2)!\,(n-|I|+1)!}.$$ each hyperplane $\{x \in \mathbb{R}^n \ |\ x_i-x_j=0\}\big\}$ having its own weight $a_{H_{i,j}}$ this time.

\smallskip

\noindent An \emph{edge} of a hyperplane arrangement $\mathcal{A}$ is a nonempty intersection of some of its hyperplanes. Denote the set of all edges of $\mathcal{A}$ by $L(\mathcal{A})$. The \emph{weight} $\mathsf{a}(E)$ of an edge $E$ is
$$\mathsf{a}(E) := \prod_{\substack{H \in \mathcal{A} \\ E \subseteq H}} a_H.$$
The \textit{multiplicity} $l(E)$ of an edge $E$ is a positive integer computed as follows \cite[2. The Nullspace of the $B$ Matrices]{DeHa}: First choose a hyperplane $H$ of $\mathcal{A}$ containing $E$. Then $l(E)$ is half the number of chambers $C$ of $\mathfrak{C}(\mathcal{A})$ which have the property that $E$ is the minimal intersection containing $\overline{C} \cap H$.\\
\noindent The formula of the determinant of a hyperplane arrangement $\mathcal{A}$ \cite[(1.1) Theorem]{Va} due to Varchenko is  $$\det \mathcal{A} = \prod_{E \in L(\mathcal{A})} \big(1- \mathsf{a}(E)^2\big)^{l(E)}.$$

\noindent From this formula, we see that we can get a more explicit or computable value of the determinant of an arrangement if we have computable forms of the $\mathsf{a}(E)$'s and the $l(E)$'s. In this article, we prove that it is the case for the arrangements associated to finite Coxeter groups. Recall that a reflection in $\mathbb{R}^n$ is a linear map sending a nonzero vector $\alpha$ to its negative while fixing pointwise the hyperplane $H$ orthogonal to $\alpha$. The finite reflection groups are also called finite Coxeter groups, since they have been classified by Coxeter \cite{Co}. Coxeter groups find applications in pratically all mathematics areas. They are particularly studied in great depth in algebra \cite{GePf}, combinatorics \cite{BjBr}, and geometry \cite{AbBr}. And they are the foundation ingredients of mathematical theories like the descent algebras, the Hecke algebras, or the Kazhdan–Lusztig polynomials. A finite Coxeter group $W$ has the presentation
$$W := \big\langle s_1, s_2, \dots, s_n \ |\ (s_i s_j)^{m_{ij}}=1 \big\rangle \quad \text{with} \quad m_{ii} = 1,\, m_{ij} \geq 2, \, m_{ij} = m_{ji}.$$
We begin with some definitions relating to the finite Coxeter group $W$.\\
The elements of the set $S := \{s_1, s_2, \dots, s_n\}$ are called the simple reflections of $W$. The set of all reflections of $W$ is denoted by $T := \{s_i^x \ |\ s_i \in S,\, x \in W\}$.\\
As usually, for a subset $J$ of $S$, $W_J$ is the parabolic subgroup $\langle J \rangle$ of $W$, $T_J$ the set of reflections $T \cap W_J$ of $W_J$, $X_J$ the set of coset representatives of minimal length of $W_J$, and $[J]$ its Coxeter class that is the set of $W$-conjugates of $J$ which happen to be subsets of $S$.

\smallskip

\noindent One says that $J$ is irreducible if the relating parabolic subgroup $W_J$ is irreducible.\\
For a subset $K$ of $J$, we write $N_{W_J}(W_K)$ for the normalizer of $W_K$ in $W_J$, and $X(J,K)$ for the set of double coset representatives $$X(J,K) := \{w \in W_J \cap X_K \cap X_K^{-1} \ |\ K^w = K\}.$$
The normalizers of the parabolic subgroups were determined by Howlett \cite{Ho}. He also proved that $N_{W_J}(W_K) = W_K \cdot X(J,K)$ \cite[Corollary 3]{Ho}.

\smallskip

\noindent Let $s_{i_1} \dots s_{i_l}$ be a reduced expression of an element $x$ of $W$. The support of $x$ is $$J(x) := \{s_{i_1}, \dots, s_{i_l}\}$$ which is independent of the choice of the reduced expression \cite[Proposition 2.16]{AbBr}. One says that $x$ has full support if $J(x) = S$.\\  
We finish with the set $\lfloor x \rceil := \{y \in W\ |\ J(y) = J(x)\ \text{and $y$ is conjugate to $x$}\}$.

\smallskip

\noindent Now, we come to the hyperplane arrangement associated to a finite Coxeter group called Coxeter arrangement. Let $H_t$ be the hyperplane $\ker(t-1)$ of $\mathbb{R}^n$ whose points are fixed by each element $t$ of $T$. The hyperplane arrangement associated to the finite Coxeter group $W$ is $\mathcal{A}_W := \{H_t\}_{t \in T}$. In this case, explicit formulas for $\mathsf{a}(E)$ and $l(E)$ can be given, and it is the aim of this article. For a subset $U$ of $T$, we write $E_U$ for the edge $$E_U := \bigcap_{u \in U}H_u.$$ The aim of this article is to prove the following result.

\begin{theorem} \label{Th}
Let $E$ be an edge of $\mathcal{A}_W$. Then $l(E) \neq 0$ if and only if there exist an irreducible subset $J$ of $S$ and an element $w$ of $W$ such that $E = E_{T_J^w}$. In this case, we clearly have $\mathsf{a}(E_{T_J^w}) = \prod_{u \in T_J}a_{H_{u^w}}$. Moreover, let $t_J$ be a reflection with support $J$, $s_J$ a simple reflection and $v$ an element of $W$ such that a reduced expression of $t_J$ is $s_J^v$. Then $$l(E_{T_J^w}) = |\lfloor t_J \rceil| \cdot |[J]| \cdot |X(S,J)| \cdot |X(J,\{s_J\})|.$$ 
\end{theorem}

\noindent Let $Y_J$ be the set of the cosets of $N_W(W_J)$, and $\mathcal{I}(S)$ the set of the Coxeter classes $[J]$ of $W$ such that $J$ is irreducible. We deduce that the Varchenko determinant of the arrangement associated to a finite Coxeter group $W$ is
$$\prod_{[J] \in \mathcal{I}(S)} \prod_{w \in Y_J}\Big(1 - \prod_{u \in T_J}a_{H_{u^w}}^2\Big)^{|\lfloor t_J \rceil| \cdot |[J]| \cdot |X(S,J)| \cdot |X(J,\{s_J\})|}.$$
We compute the Varchenko determinants associated to the irreducible Coxeter groups in Section \ref{Full}. We deduce the Varchenko determinant associated to a finite Coxeter group from the following remark: If $W = W_1 W_2$, where $W_1$ and $W_2$ are two irreducible Coxeter groups, then
$$\det \mathcal{A}_W = (\det \mathcal{A}_{W_1})^{|W_2|} (\det \mathcal{A}_{W_2})^{|W_1|}.$$ 

\noindent Recall that there is a one-to-one correspondence between the elements of $W$ and the chambers of $\mathcal{A}_W$ such that: if the chamber $C$ corresponds to the neutral element $e$ and $C_x$ to another element $x$, then $C_x = Cx$ with $x = t_1 \dots t_r$, the $H_{t_i}$'s being the hyperplanes one goes through from $C$ to $C_x$ \cite[Theorem 1.69]{AbBr}.\\ Let $\langle \overline{C}_x \cap H_t \rangle$ be the subspace generated by the closed face $\overline{C}_x \cap H_t$ of the chamber $C_x$. Determining the multiplicity of the edge $E$ contained in the hyperplane $H_t$ consists of counting the half of the chambers $C_x$ which have the property that $E$ is the minimal edge containing $\langle \overline{C}_x \cap H_t \rangle$. For the proof of Theorem \ref{Th}, we need to introduce the three propositions that we prove in the next three sections. 

\begin{proposition} \label{Pro1}
Let $E$ be an edge of $\mathcal{A}$. Then $l(E) \neq 0$ if and only if there exist an irreducible subset $J$ of $S$ and an element $w$ of $W$ such that $E = E_{T_J^w}$.
\end{proposition} 

\noindent Let us introduce the set $L(E, t) := \big\{x \in W\ |\ \langle \overline{C}_x \cap H_t \rangle = E\big\}$.

\begin{proposition} \label{Pro2}
Let $J$ be an irreducible subset of $S$, $t$ a reflection with support $J$, $s$ a simple reflection, and $v$ an element of $W$ such that a reduced expression of $t$ is $s^v$. For a conjugate $K$ of $J$, let $c_{K,J}$ be an element of $W$ such that $K^{c_{K,J}} = J$, and for a conjugate $u$ of $t$ with support $J$, let $c_{u,t}$ be an element of $W$ such that $u^{c_{u,t}} = t$. Then, $$L(E_{T_J}, t) = \bigsqcup_{K \in [J]} c_{K,J} X(S,J) \Big(\bigsqcup_{u \in \lfloor t \rceil} c_{u,t} N_{W_J}(W_{\{s\}})^v \Big).$$
\end{proposition}

\begin{proposition} \label{Pro3}
Let $u,v \in T$ and let $E$ be an edge contained in both $H_u$ and $H_v$. Then $$|L(E, u)| = |L(E, v)|.$$
\end{proposition}

\section{The Coxeter Complex}

\noindent Not all edges are relevant, or in other words, there are some edges whose multiplicities are null. We develop the condition for an edge $E$ to be relevant which means $l(E) \neq 0$. 

\begin{lemma} \label{LeIr}
A finite Coxeter group $W$ is irreducible if and only if $W$ has a reflection of full support.
\end{lemma}

\begin{proof}
If $W$ is irreducible, then $W$ has a highest root \cite[2.10 Construction of root systems, 2.13 Groups of types $H_3$ and $H_4$]{Hu}, and the reflection corresponding to the highest root has full support.\\
Suppose that $W$ is the product of nontrivial Coxeter groups $W_1$ and $W_2$, and let $t$ be a reflection in $W$. Without loss of generality, we can suppose that $t$ is a conjugate of a simple reflection $s$ of $W_1$, hence $t$ lies in $W_1$ and can not have full support.
\end{proof}

\noindent We continue our investigation by using the Coxeter complex. Recall that the Coxeter complex $\mathcal{C}$ of $W$ is the set $\{W_Jx\ |\ J \subseteq S,\, x \in W\}$ of faces. The Coxeter complex is a combinatorial setup which permits to study the geometrical structure of $\mathcal{A}_W$. Indeed, $\mathbb{R}^n$ is partitioned by $\mathcal{C}$ \cite[1.15 The Coxeter complex]{Hu}.\\
The chamber $C_x$ is identified with the singleton $\{x\}$. The coset $W_J w$ is a face of $\{x\}$ if and only if $W_J w = W_J x$. Then the closure of $C_x$ is
\begin{equation} \label{Cx}
\overline{C}_x := \{W_Jx\ |\ J \subseteq S\}.
\end{equation}
More generally, $W_K w$ is a face of $W_J x$ if and only if $W_J x \subseteq W_K w$ and $W_K w = W_K x$. Then the closure of a face $W_Jx$ of $C_x$ is
\begin{equation} \label{WJx}
\overline{W_Jx} := \{W_Kx\ |\ J \subseteq K \subseteq S\}.
\end{equation}
Hyperplanes, and more generally egdes, can be described as collections of the faces they consist of. We extend the definition of $J(x)$ to a subset $X$ of $W$ with the following way: $$J(X) := \bigcup_{x \in X} J(x).$$

\begin{lemma} \label{LeHE}
Let $t \in T$. Then $H_t = \{W_J w\ |\ J \subseteq S,\, w \in W,\, J(t^{w^{-1}}) \subseteq J\}$.
\end{lemma}

\begin{proof}
As $W$ acts by right multiplication, $W_J w \in H_t$ if and only if $W_J w t = W_J w$, i.e. $W_J^w t = W_J^w$, i.e. $t \in W_J^w$, i.e. $t^{w^{-1}} \in W_J$, i.e. $J(t^{w^{-1}}) \subseteq J$.
\end{proof}

\begin{lemma} \label{LeIn}
Let $x \in W$ and $t \in T$. Then, $\overline{C}_x \cap H_t = \overline{W_{J(t^{x^{-1}})}x}$.
\end{lemma}

\begin{proof}
We have 
\begin{align*}
\overline{C}_x \cap H_t\ =\ & \{W_Jx\ |\ J \subseteq S\} \cap \{W_J w\ |\ J \subseteq S,\, w \in W,\, J(t^{w^{-1}}) \subseteq J\} \\
&\ \text{(Equation \ref{Cx} and Lemma \ref{LeHE})} \\ 
=\ &\ \{W_Jx\ |\ J(t^{x^{-1}}) \subseteq J\} \\
=\ &\ \overline{W_{J(t^{x^{-1}})}x} \quad \text{(Equation \ref{WJx})}
\end{align*}
\end{proof}

\begin{lemma} \label{LeFa}
Let $W_Jx$ be a face of the Coxeter complex. Then, the subspace $\langle W_Jx \rangle$ generated by $W_Jx$ is the edge $E_{T_J^x}$. 
\end{lemma}

\begin{proof}
It is clear that the subspace generated by $\overline{W_Jx}$ is equal to the subspace generated by $W_Jx$.
From the proof of Lemma \ref{LeHE}, we know that $H_t$ is a hyperplane containing $W_Jx$ is equivalent to $J(t^{x^{-1}}) \subseteq J$ which is equivalent to $t \in T_J^x$.
\end{proof}

\noindent Let us take a reflection $t$ of $U$. Recall that $l(E_U)$ is half the number of chambers $C_x$ which have the property that $E_U$ is the minimal intersection containing $\langle \overline{C}_x \cap H_t \rangle$. Minimality implies equality that is, for the chamber $C_x$ to be counted, we must have $\langle \overline{C}_x \cap H_t \rangle = E_U$ i.e. $\langle \overline{W_{J(t^{x^{-1}})}x} \rangle = E_U$ (Lemma \ref{LeIn}) i.e. $E_{T_{J(t^{x^{-1}})}^x} = E_U$ (Lemma \ref{LeFa}). Moreover, since $t^{x^{-1}}$ is a full support reflection of the group $W_{J(t^{x^{-1}})}$, we know from Lemma \ref{LeIr} that $J(t^{x^{-1}})$ is an irreducible subset of $S$. So $E_U$ is of the form $E_J^x$, where $J$ is irreducible, otherwise $l(E_U) = 0$. That proves Proposition \ref{Pro1}.

\section{The Chambers to Consider}

\noindent For a relevant edge $E$ and a hyperplane $H_t$ containing $E$, we determine the chambers $C_x$ such that $\langle \overline{C}_x \cap H_t \rangle = E$.

\begin{lemma} \label{LeDe}
Let $J$ be an irreducible subset of $S$, $t$ a reflection with support $J$, and $E$ the relevant edge $E_{T_J}$. For each $K$ in the Coxeter class of $J$, let $c_{K,J}$ be the coset of minimal length of $N_W(W_J)$ such that $K^{c_{K,J}} = J$. Then, $$L(E, t) = \bigsqcup_{K \in [J]} c_{K,J}\, \big\{x \in N_W(W_J)\ \big|\ J(t^{x^{-1}}) = J\big\}.$$
\end{lemma}

\begin{proof}
We have
\begin{align*}
L(E, t)\ &\ = \big\{x \in W \ |\ \langle \overline{C}_x \cap H_t \rangle = E\big\} \\
\ &\ = \big\{x \in W \ |\ \langle \overline{W_{J(t^{x^{-1}})}x} \rangle = E\big\} \quad \text{(Lemma \ref{LeIn})} \\
\ &\ = \big\{x \in W \ |\ E_{T_{J(t^{x^{-1}})}^x} = E_{T_J}\big\} \quad \text{(Lemma \ref{LeFa})} 
\end{align*}
Denoting $J(t^{x^{-1}})$ by $K$, the equality $E_{T_{J(t^{x^{-1}})}^x} = E_{T_J}$ means that
\begin{itemize}
\item[$\bullet$] $|J(t^{x^{-1}})| = |J|$,
\item[$\bullet$] $K \in [J]$,
\item[$\bullet$] and $x \in c_{K,J}N_W(W_J)$.
\end{itemize}
Hence $$L(E, t) = \bigsqcup_{K \in [J]} \big\{x \in c_{K,J}N_W(W_J)\ \big|\ |J(t^{x^{-1}})| = |J|\big\}.$$
Let $x = c_{K,J}y$ with $y \in N_W(W_J)$. Since $|J(t^{x^{-1}})| = |J|$ if and only if $|J(t^{y^{-1}})| = |J|$ if and only if $J(t^{y^{-1}}) = J$, we obtain the result.
\end{proof}

\noindent Recall that $X(S,J) := \{x \in X_J \ |\ J^x = J\}$. We introduce the set $$W_J(t) := \{x \in W_J \ |\ J(t^{x^{-1}}) = J\}.$$

\begin{lemma} \label{LeSe}
Let $J$ be a irreducible subset of $S$. Consider a reflection $t$ of $W_J$ with support $J$. Then,
$$\big\{x \in N_W(W_J)\ \big|\ J(t^{x^{-1}}) = J\big\}\ =\ X(S,J) \cdot W_J(t).$$
\end{lemma}

\begin{proof}
We have $N_W(W_J) = W_J \cdot X(S,J)$ \cite[Corollary 3]{Ho}. So
$$\big\{x \in N_W(W_J)\ \big|\ J(t^{x^{-1}}) = J\big\}\ =\ \big\{yz\ \big|\ y \in W_J,\, z \in X(S,J),\, J(t^{z^{-1} y^{-1}}) = J\big\}.$$
Since we always have $t^{z^{-1}} = t$, the remaining condition is $J(t^{y^{-1}}) = J$. 
\end{proof}

\noindent We write $C(x)$ for the centralizer of the element $x$ of $W$.

\begin{lemma} \label{LeWj}
Let $J$ be a irreducible subset of $S$. Consider a reflection $t = s^v$ with support $J$ where $s$ is a simple reflection and $v$ and element of $W_J$. For another reflection $u$ of $W_J$ with support $J$ and conjugate to $t$, let $c_{u,t}$ be 
the coset of minimal length of $C(t)$ such that $u^{c_{u,t}} = t$. Then,
$$W_J(t) = \bigsqcup_{u \in \lfloor t \rceil} c_{u,t} N_{W_J}(W_{\{s\}})^v.$$
\end{lemma}

\begin{proof}
The equality $J(t^{x^{-1}}) = J$ means that $x^{-1} \in C(t)c_{t,u}$ or $x \in c_{u,t}C(t)$, where $u$ a conjugate of $t$ with support $J$. Then, $$W_J(t) = \bigsqcup_{u \in \lfloor t \rceil} c_{u,t}C(t).$$ Since $C(t) = C(s)^v$ and $C(s) = N_{W_J}(W_{\{s\}})$, we get the result.
\end{proof}

\noindent Proposition \ref{Pro2} is a combination of Lemma \ref{LeDe}, Lemma \ref{LeSe}, and Lemma \ref{LeWj}.

\section{Invariance of the Multiplicity}

\noindent For the calculation of the multiplicity of an edge $E$, one has to choose a hyperplane containing $E$. We prove that the result of the calculation is independent of the choice of the hyperplane containing $E$.

\begin{lemma} \label{LemConv}
Consider two hyperplanes $H_u$ and $H_v$ of $\mathcal{A}_W$ associated with the orthogonal vectors $\alpha_u$ and $\alpha_v$ respectively, of equal length and from the positive root system . Let $r$ be the rotation of $\mathbb{R}^n$, not necessarily in $W$, which transforms $\alpha_u$ to $\alpha_v$. Then, $r(\mathcal{A}_W) = \mathcal{A}_W$.
\end{lemma}

\begin{proof}
Let $E_{u, v}$ be the $2$-dimensional subspace $\langle \alpha_u, \alpha_v \rangle$:
\begin{itemize}
\item[$\bullet$] on $E_{u, v}$, the map $r$ is the rotation of angle $\theta = \arccos \frac{\alpha_u \,.\, \alpha_v}{\parallel\alpha_u\parallel \, \parallel\alpha_v\parallel}$,
\item[$\bullet$] on $E_{u, v}^{\perp}$, the map $r$ is the identity map.
\end{itemize}
We have $\mathcal{A}_W = \mathcal{A}_1 \sqcup \mathcal{A}_2$ with $$\mathcal{A}_1 := \{H \in \mathcal{A}\ |\ E_{u, v} \subset H\} \quad \text{and} \quad \mathcal{A}_2 := \{H \in \mathcal{A}\ |\ \dim E_{u, v} \cap H = 1\}.$$
\begin{itemize}
\item[$\circ$] For all $H$ in $\mathcal{A}_1$, we have $r(H) = H$.
\item[$\circ$] Let $p$ be projection on the subspace $E_{u, v}$. The arrangement $p(\mathcal{A}_2)$ is the arrangement of a dihedral group whose angle between two certain hyperplanes is $\theta$. Then, $r\big(p(\mathcal{A}_2) \big) = p(\mathcal{A}_2)$. Hence, for any $H$ in $\mathcal{A}_2$, we have
$$r(H) = r\big(p(H) \oplus E_{u, v}^{\perp}\big) = r\big(p(H)\big) \oplus E_{u, v}^{\perp}$$
which still belongs to $\mathcal{A}_2$.
\end{itemize}
\end{proof}

\noindent We prove Proposition \ref{Pro3} now. Consider two hyperplanes $H_u$ and $H_v$ containing the edge $E$. Let $\alpha_u$ and $\alpha_v$ be the unit vectors of the positive root system associated to $W$ which are orthogonal to $H_u$ and $H_v$ respectively. We use the rotation $r$ of Lemma \ref{LemConv} transforming $\alpha_u$ to $\alpha_v$, and leaving $\mathcal{A}$ invariant. For a given $w$ in $L(E, u)$, we have  
\begin{align*}
\overline{C}_w \cap H_u & = \overline{C}_w \cap E\\
(\overline{C}_w \cap H_u)r & = (\overline{C}_w \cap E)r\\
(\overline{C}_w)r \cap (H_u)r & = (\overline{C}_w)r \cap (E)r\\
(\overline{C}_w)r \cap H_v & = (\overline{C}_w)r \cap E.
\end{align*}
Hence $L(E, u)r = L(E, v)$, and $|L(E, u)| = |L(E, v)|$ which is Proposition \ref{Pro3}.

\section{The Multiplicity of an Edge} \label{Proof}

\noindent We establish a formula for the multiplicity of a relevant edge of $\mathcal{A}_W$ in this section.

\smallskip

\noindent We begin with the proof of Theorem \ref{Th}. From Proposition \ref{Pro1}, we know that the relevant edges are the intersections of hyperplanes $E_{T_J^w}$ with the condition that $J$ is irreducible. The weight of $E_{T_J^w}$ is obviously $\prod_{u \in T_J}q_{u^w}$ so that the real problem concerns $l(E_{T_J^w})$. Fixing a hyperplane $H_t$ containing $E_{T_J^w}$, the set of chambers taken into account to determine $l(E_{T_J^w})$ is $L(E_{T_J^w}, t)$. But Proposition \ref{Pro3} allows us to choose any hyperplane containing $E_{T_J^w}$. Hence the multiplicity of $E_{T_J^w}$ is $$l(E_{T_J^w}) = \frac{1}{2}|L(E_{T_J^w}, t)|.$$

\begin{lemma} \label{LeLw}
Let $J$ be an irreducible subset of $S$, and $w$ an element of $W$. Then, $$L(E_{T_J^w}, t^w) = L(E_{T_J}, t) w.$$
\end{lemma}

\begin{proof}
Let $x \in L(E_{T_J}, t)$. We have
\begin{align*}
\overline{C}_x \cap H_t & = \overline{C}_x \cap E_{T_J} \\
(\overline{C}_x \cap H_t)w & = (\overline{C}_x \cap E_{T_J})w \\
(\overline{C}_x)w \cap (H_t)w & = (\overline{C}_x)w \cap (E_{T_J})w \\
\overline{C}_{xw} \cap H_{t^w} & = \overline{C}_{xw} \cap E_{T_J^w} 
\end{align*}
Then $L(E_{T_J}, t) w \subseteq L(E_{T_J^w}, t^w)$. With the same argument, we obtain also $L(E_{T_J^w}, t^w) w^{-1} \subseteq L(E_{T_J}, t)$. Hence $L(E_{T_J^w}, t^w) = L(E_{T_J}, t) w$.
\end{proof}

\noindent We see in Lemma \ref{LeLw} that we just need to investigate $L(E_{T_J}, t)$ for each Coxeter class $[J]$. From Proposition \ref{Pro2}, we get 
\begin{align*}
l(E_{T_J}) & = \frac{1}{2}|L(E_{T_J}, t)|\\
& = \frac{1}{2} \Big|\bigsqcup_{K \in [J]} c_{K,J} X(S,J) \Big(\bigsqcup_{u \in \lfloor t \rceil} c_{u,t} N_{W_J}(W_{\{s\}})^v \Big)\Big|\\
& = |\lfloor t \rceil| \cdot |[J]| \cdot |X(S,J)| \cdot |X(J,\{s\})| \quad \text{since} \quad \frac{1}{2}|N_{W_J}(W_{\{s\}})| = |X(J,\{s\})|,
\end{align*}
which finishes the proof of Theorem \ref{Th}.

\section{Computing the Determinants of Finite Coxeter Groups} \label{Full}

\noindent Before computing the Varchenko determinants of the irreducible finite Coxeter groups, we first have to determine their numbers of full support reflections.

\smallskip

\noindent The graph $\Gamma = (S,M)$ associated to a finite Coxeter group $W$ is defined as follows:
\begin{itemize}
\item the vertex set $S$ is the simple reflection set of $W$,
\item the edge set $M$ is composed by the pairs of simple reflections $\{s_i, s_j\}$ such that $m_{ij} \geq 3$, and these edges $\{s_i, s_j\}$ are labeled by $m_{ij}$.
\end{itemize}
We only consider the graphs associated to the irreducible finite Coxeter groups \cite[Table 1.1]{GePf}. The reflections of an irreducible finite Coxeter group $W$ form a single conjugacy class if and only if the edge labels of its associated graph are all odd \cite[n$^{\circ}$1.3, Proposition 3]{Bo}. It is the case of the Coxeter groups $A_{n-1}$, $D_n$, $E_6$, $E_7$, $E_8$, $H_3$, $H_4$, and $I_2(m)$ ($m$ odd).\\
For the case of the Coxeter groups $B_n$, $F_4$, and $I_2(m)$ ($m$ even), removing the even labeled edge from their associated graphs let two graphs of type $A$ left. So their reflections form two conjugacy classes.

\smallskip

\noindent We can see the numbers of reflections of the irreducible finite Coxeter groups in the book of Bj\"orner and Brenti \cite[Appendix A1]{BjBr} for example. Using the Principle of Inclusion and Exclusion, applied to the irreducible maximal parabolic subgroups, and the Pascal’s triangle in the form $\binom{n}{2} - \binom{n-1}{2} = n-1$, we get the number of full support reflections of the irreducible finite Coxeter groups in Table \ref{FuSuRe}.

\begin{table}
\begin{center}
    \begin{tabular}{ p{2.7cm} | p{2.2cm}  p{2.2cm}  p{5cm} }
    \bf{Coxeter Groups} & \bf{Number of Reflections} & \bf{Number of Conjugacy Classes} & \bf{Number of Full Support Reflections} \\ \hline
    $A_{n-1}$ & $\binom{n}{2}$ & $1$ & $1$ \\ 
    $B_n$ & $n^2$ & $2$ & $1\ |\ n-1$ \\ 
    $D_n$ & $n(n-1)$ & $1$ & $n-2$ \\
    $E_6$ & $36$ & $1$ & $7$ \\
    $E_7$ & $63$ & $1$ & $16$ \\
    $E_8$ & $120$ & $1$ & $44$ \\
    $F_4$ & $24$ & $2$ & $5\ |\ 5$ \\
    $H_3$ & $15$ & $1$ & $8$ \\
    $H_4$ & $60$ & $1$ & $42$  \\
    $I_2(m)$ odd/even & $m$ / $m$ & $1$ / $2$ & $m-2$ $\ $ / $\ $ $\frac{m-2}{2}\ |\ \frac{m-2}{2}$ 
    \end{tabular}
\end{center}
\caption{Number of Full Support Reflections of the Irreducible Finite Coxeter Groups.} \label{FuSuRe}
\end{table}

\smallskip

\noindent We are now able to compute the Varchenko determinant of a finite Coxeter group by using Theorem \ref{Th}. The necessary ingredients mentioned to calculate this determinant are exposed in Table \ref{VaDe} for all irreducible finite Coxeter groups. They are obtained with tools of Table \ref{FuSuRe}, those in \cite[Proposition 2.3.8, 2.3.10, 2.3.13]{GePf}, and \cite[Table A.1, A.2]{GePf}.\\
\noindent Let $[\pm n] := \{-n, \dots, -2, -1, 1, 2, \dots, n\}$. We write $\overline{2^{[\pm n]}}$ for the subset of $2^{[\pm n]}$ having the following properties:
\begin{itemize}
\item[$\bullet$] the elements of $\overline{2^{[\pm n]}}$ are the elements $\{i_1, \dots, i_t\}$ of $2^{[\pm n]}$ such that $|i_r| \neq |i_s|$ if $r \neq s$,
\item[$\bullet$] and if $\{i_1, \dots, i_t\} \in \overline{2^{[\pm n]}}$, then $\{-i_1, \dots, -i_t\} \notin \overline{2^{[\pm n]}}$.
\end{itemize}
Using Theorem \ref{Th} and Table \ref{VaDe}, we refind, for example, the determinants 
\begin{align*}
\det \mathcal{A}_{B_n} = & \prod_{\substack{J \in \overline{2^{[\pm n]}} \\ |J| \geq 2}} \Big(1- \prod_{\{i,j\} \in \binom{J}{2}} a_{H_{i,j}}^2\Big)^{2^{n-|J|+1}\,(|J|-2)!\,(n-|J|+1)!} \\
& \prod_{\substack{I \in 2^{[n]} \\ |I| \geq 1}} \Big(1- \prod_{i \in I} a_{H_i}^2
\prod_{\{i,j\} \in \binom{I}{2}} a_{H_{i,j}}^2\, a_{H_{-i,j}}^2\Big)^{2^{n-1}\,(|I|-1)!\,(n-|I|)!},
\end{align*}
computed by Randriamaro \cite{Ra} with combinatorial methods.

\begin{table}
\begin{center}
    \begin{tabular}{ p{1.7cm} | p{2.3cm}  p{1.6cm}  p{1.3cm}  p{3cm}  p{3cm} }
    \bf{Coxeter Groups} & $J$ & $|\lfloor t_J \rceil|$ & $|[J]|$ & $|X(S,J)|$ & $|X(J,\{s_J\})|$ \\ \hline
    $A_{n-1}$ & $A_i$ & $1$ & $n-i$ & $(n-i-1)!$ & $(i-1)!$ \\ \hline
    $B_n$ & $A_i$ $(i \leq n-1)$ & $1$ & $n-i$ & $2^{n-i}(n-i-1)!$ & $(i-1)!$ \\
          & $B_j$ & $1\ |\ j-1$ & $1$ & $2^{n-1}(n-j)!$ & $(j-1)!\ |\ (j-2)!$ \\ \hline
    $D_n$ & $A_i$ $(i \leq n-1)$ & $1$ & $n-i+1$ & $2^{n-i-1}(n-i-1)!$ & $(i-1)!$ \\
          & $D_j$ & $j-2$ & $1$ & $2^{n-j}(n-j)!$ & $2^{j-2}(j-2)!$ \\ \hline
    $E_6$ & $A_1$ & $1$ & $6$ & $720$ & $1$ \\
          & $A_2$ & $1$ & $5$ & $72$ & $1$ \\
          & $A_3$ & $1$ & $5$ & $8$ & $2$ \\    
          & $A_4$ & $1$ & $4$ & $2$ & $6$ \\
          & $D_4$ & $2$ & $1$ & $6$ & $8$ \\        
          & $A_5$ & $1$ & $1$ & $2$ & $24$ \\       
          & $D_5$ & $3$ & $2$ & $1$ & $48$ \\          
          & $E_6$ & $7$ & $1$ & $1$ & $720$ \\ \hline
    $E_7$ & $A_1$ & $1$ & $7$ & $23040$ & $1$ \\
          & $A_2$ & $1$ & $6$ & $1440$ & $1$ \\
          & $A_3$ & $1$ & $6$ & $96$ & $2$ \\    
          & $A_4$ & $1$ & $5$ & $12$ & $6$ \\
          & $D_4$ & $2$ & $1$ & $48$ & $8$ \\        
          & $A_5'$ & $1$ & $1$ & $12$ & $24$ \\ 
          & $A_5''$ & $1$ & $1$ & $4$ & $24$ \\
          & $D_5$ & $3$ & $2$ & $4$ & $48$ \\
          & $A_6$ & $1$ & $1$ & $2$ & $120$ \\      
          & $D_6$ & $4$ & $1$ & $2$ & $384$ \\
          & $E_6$ & $7$ & $1$ & $2$ & $720$ \\         
          & $E_7$ & $16$ & $1$ & $1$ & $23040$ \\ \hline
    $E_8$ & $A_1$ & $1$ & $8$ & $2903040$ & $1$ \\
          & $A_2$ & $1$ & $7$ & $103680$ & $1$ \\
          & $A_3$ & $1$ & $7$ & $3840$ & $2$ \\    
          & $A_4$ & $1$ & $6$ & $240$ & $6$ \\ 
          & $D_4$ & $2$ & $1$ & $1154$ & $8$ \\       
          & $A_5$ & $1$ & $4$ & $24$ & $24$ \\ 
          & $D_5$ & $3$ & $2$ & $48$ & $48$ \\
          & $A_6$ & $1$ & $3$ & $4$ & $120$ \\
          & $D_6$ & $4$ & $1$ & $8$ & $384$ \\
          & $E_6$ & $7$ & $1$ & $12$ & $720$ \\ 
          & $A_7$ & $1$ & $1$ & $2$ & $720$ \\      
          & $D_7$ & $5$ & $1$ & $2$ & $3840$ \\                  
          & $E_7$ & $16$ & $1$ & $2$ & $23040$ \\
          & $E_8$ & $44$ & $1$ & $1$ & $2903040$ 
    \end{tabular}
    \end{center}
\end{table}

\begin{table}
\begin{center}
    \begin{tabular}{ p{1.7cm} | p{2.3cm}  p{1.6cm}  p{1.3cm}  p{3cm}  p{3cm} } \hline
    $F_4$ & $A_1'$ & $1$ & $2$ & $48$ & $1$ \\
          & $A_1''$ & $1$ & $2$ & $48$ & $1$ \\
          & $A_2'$ & $1$ & $1$ & $12$ & $1$ \\
          & $A_2''$ & $1$ & $1$ & $12$ & $1$ \\
          & $B_2$ & $2$ & $1$ & $8$ & $2$ \\
          & $B_3'$ & $1\ |\ 2$ & $1$ & $2$ & $8\ |\ 4$ \\
          & $B_3''$ & $1\ |\ 2$ & $1$ & $2$ & $8\ |\ 4$ \\    
          & $F_4$ & $10$ & $1$ & $1$ & $48$ \\  \hline
    $H_3$ & $A_1$ & $1$ & $3$ & $4$ & $1$ \\
          & $A_2$ & $1$ & $1$ & $2$ & $1$ \\
          & $I_2(5)$ & $3$ & $1$ & $2$ & $1$ \\    
          & $H_3$ & $8$ & $1$ & $1$ & $4$ \\  \hline 
    $H_4$ & $A_1$ & $1$ & $4$ & $120$ & $1$ \\
          & $A_2$ & $1$ & $2$ & $12$ & $1$ \\  
          & $I_2(5)$ & $3$ & $1$ & $20$ & $1$ \\
          & $A_3$ & $1$ & $1$ & $2$ & $2$ \\   
          & $H_3$ & $8$ & $1$ & $2$ & $4$ \\        
          & $H_4$ & $42$ & $1$ & $1$ & $120$ \\  \hline
    $I_2(m)$ & $A_1$ odd/even & $1$ / $2$ & $2$ / $1$ & $1$ & $1$ \\
          & $I_2(m)$ & $m-2$ & $1$ & $1$ & $1$                          
    \end{tabular}
    \end{center}
\caption{Multiplicities of the Coxeter Classes.} \label{VaDe}    
\end{table}

\bibliographystyle{abbrvnat}

\end{document}